\documentclass[12pt,reqno]{amsart}

\usepackage{amssymb, amsmath, amsthm}
\usepackage[colorlinks=true, urlcolor=blue, linkcolor=blue, citecolor=blue]{hyperref}
\usepackage[alphabetic,lite,nobysame]{amsrefs}
\usepackage{verbatim}
\usepackage{amscd}   % for commutative diagrams
\usepackage[all]{xy} % for complicated commutative diagrams
\usepackage{youngtab} % for Young tableaux
\usepackage{young} % for Young tableaux
\usepackage{ytableau}
\usepackage{tikz}
\usepackage{ mathrsfs }
\usepackage{cases}
\usepackage{array}
\usepackage{cellspace}
\usepackage{tabu}
\usepackage{calligra,mathrsfs}
\usepackage{bm}
\usepackage{mathtools}
\usepackage{tikz-cd}

\usepackage[margin=1in]{geometry}

\DeclareMathOperator{\ShHom}{\mathscr{H}\text{\kern -3pt {\calligra\large om}}\,}

\newcommand{\CC}{\mathbb{C}}

\newcommand{\m}{\mathfrak{m}}

\newcommand{\Tor}{\operatorname{Tor}}

\newcommand{\Sym}{\operatorname{Sym}}

\newcommand{\coker}{\operatorname{coker}}

\renewcommand{\ker}{\operatorname{ker}}

\newcommand{\reg}{\operatorname{reg}}

\newcommand{\mc}[1]{\mathcal{#1}}

\def\PP{{\mathbb P}}
\def\lra{\longrightarrow}

\newtheorem{theorem}{Theorem}[section]
\newtheorem*{theorem*}{Theorem}
\newtheorem*{problem*}{Problem}
\newtheorem{lemma}[theorem]{Lemma}
\newtheorem{conjecture}[theorem]{Conjecture}

\newtheorem*{corollary*}{Corollary}

\theoremstyle{definition}

\newtheorem*{definition*}{Definition}
\newtheorem{example}[theorem]{Example}

\theoremstyle{remark}

\newtheorem*{remark*}{Remark}

\numberwithin{equation}{section}

\begin{document}

\title[Higher syzygy bundles and the Eisenbud--Huneke--Ulrich conjecture]{Higher syzygy bundles and \\ the Eisenbud--Huneke--Ulrich conjecture}
\author{Fuxiang Yang}
\address{Department of Mathematics, University of Notre Dame, 255 Hurley, Notre Dame, IN 46556}
\email{fyang6@nd.edu}


\date{\today}


\begin{abstract} 
We study ideals in a polynomial ring with partially linear (virtual) resolutions. We establish an effective bound beyond which their powers coincide with a power of the maximal ideal, proving a slightly weaker version of the Eisenbud--Huneke--Ulrich conjecture for a more general class of ideals. We also obtain an upper bound on their Castelnuovo--Mumford regularity, extending Macaulay’s bound and a theorem of Eisenbud--Huneke--Ulrich. We introduce higher syzygy bundles, which generalize the classical Green--Lazarsfeld syzygy bundle. The key ingredient is the relationship between the syzygies of partially linear ideals and the sheaf cohomology of higher syzygy bundles.
\end{abstract}
\maketitle
\vspace{-0.1in}
\section{Introduction}\label{sec:intro}
Let $S = \CC[x_0,x_1,\dots,x_n]$ be the polynomial ring, $\m$ the maximal ideal generated by the variables, and $I$ a homogeneous ideal generated in degree $d$. We say that the resolution of $I$ is \textbf{linear for $p$ steps} if its first linear strand
\begin{equation}\label{eq:firstLinearStrand}
    \cdots\lra S(-d-p-1)^{b_{p+1,d}} \lra S(-d-p)^{b_{p,d}} \lra \cdots \lra S(-d)^{b_{0,d}} \lra I \lra 0
\end{equation}
is exact up to cohomological degree $-p+1$. Equivalently, for every $0 \le i \le p$, we have
\[\beta_{i,j}(I) = \dim_\CC\left( \Tor_i^S(I,\CC)_{i+j}\right) \ne 0 \quad \iff \quad j = d.\]
This notion measures the complexity of the ideal $I$. As $p$ increases, the minimal free resolution of $I$ should become simpler. Eisenbud, Huneke, and Ulrich conjectured the following:
\begin{conjecture}[\cite{EHU}*{Conjecture~1.4}]\label{conj:EHU}
    If $I$ is $\m$-primary and the resolution of $I$ is linear for $p$ steps, where $p \ge 1$, then
    \[I^t = \m^{td} \quad \text{for all }t \ge \frac{n}{p}.\]
\end{conjecture}
\noindent The conjecture is known to hold when $I$ is a monomial ideal \cite{EHU}*{Theorem~8.1}, when $n = 2$ \cite{EHU}*{Corollary~7.7}, and when $p\ge n/2 $ \cite{EHU}*{Corollary~7.7}. However, little is known beyond these cases. In our investigation, it is natural to consider the following more general setting. We say that the resolution of $I$ is \textbf{virtually linear for $p$ steps} if the sheafification of its first linear strand (\ref{eq:firstLinearStrand}) is exact up to cohomological degree $-p+1$, see also \cite{BES}. If the resolution of $I$ is linear for $p$ steps, then it is also virtually linear for $p$ steps because sheafification preserves exactness. Our first result gives a linear effective bound in the direction of Conjecture~\ref{conj:EHU}. 
\begin{theorem}\label{thm:linearBound}
    If $I$ is $\m$-primary and the resolution of $I$ is virtually linear for $p$ steps, where $p \ge 1$, then
    \begin{equation}\label{eq:linearBound}
        I^t = \m^{td}\quad \text{ for all }t \ge  \left\lfloor\frac{n}{p+1}\right\rfloor + \left\lfloor\frac{n-1}{p}\right\rfloor.
    \end{equation}
\end{theorem}
\noindent In particular, Theorem~\ref{thm:linearBound} implies that Conjecture~\ref{conj:EHU} holds whenever $p\ge (n-1)/2$ or $n \le 3$, thereby recovering and extending the previously known cases. A key ingredient in proving Theorem~\ref{thm:linearBound} is an upper bound on the Castelnuovo--Mumford regularity of such ideals. Many partial results on regularity are known for specific classes of ideals with partially linear resolutions. For example, such bounds are known when $S/I$ is Koszul \cite[Theorem~7.1]{ACI-subadd}, when $I$ is squarefree monomial \cite[Theorem~1.1]{DaoVu}, and when $I$ is $\m$-primary monomial \cite[Theorem~3.1]{DaoEisenbud}. In this paper, we focus on the case in which $I$ is $\m$-primary. Unless otherwise stated, we assume throughout that $I$ is $\m$-primary and generated in degree $d$.

With this setup in place, assume now that the resolution of $I$ is linear for $p$ steps. In the case $p = 0$, Macaulay proved \cite{Macaulay}*{\S 86} that
\[\reg(S/I) \le (n+1)(d-1).\]
For modern treatments of this result, see \cite[Example~2.4]{EHL} and \cite[Remark~3.1]{HMNU}. At the other extreme, when $p = n$, $I$ has a linear resolution, and hence
\[\reg(S/I) = d-1.\]
More generally, when $p\ge (n-1)/2$, that is, when at least half of the minimal free resolution is linear, Eisenbud, Huneke, and Ulrich proved in \cite[Corollary~3.2]{EHU} that
\[\reg(S/I) \le 2(d-1).\]
Our second result applies to all values of $p$ and extends the theory to virtually linear ideals.
\begin{theorem}\label{thm:regularityBound}
    If the resolution of $I$ is virtually linear for $p $ steps, then
    \[\reg(S/I) \le \left\lceil \frac{n+1}{p+1}\right\rceil (d - 1).\]
    Moreover, this bound is sharp for $p = 0$, $p = 1$, and $p \ge (n-1)/2$.
\end{theorem}
\noindent In the case $p \ge (n-1)/2$, the optimal example comes from secant varieties of high degree curves studied in \cite{ENP}, see Example~\ref{ex:sharp} for the optimal examples. The bound we obtain is not unexpected. Indeed, by a direct computation, if the resolution of $I$ is linear for $p$ steps and $I$ satisfies the subadditivity condition, see \cite{M-Subadd} for the definition, then the same regularity bound is obtained. We therefore conjecture that the bound in Theorem~\ref{thm:regularityBound} is sharp in general. 

Assume the resolution of $I$ is virtually linear for $p$ steps. Sheafifying the first linear strand (\ref{eq:firstLinearStrand}), we obtain a complex
\begin{equation}\label{eq:linearStrand}
    \mc{A}^\bullet \colon W_p(-p) \overset{d_p}{\lra} W_{p-1}(-p+1) \overset{d_{p-1}}{\lra} \cdots \lra W_1(-1) \overset{d_1}{\lra} V \otimes \mc{O}_{\PP^n}\left( \overset{d_0}{\lra }\widetilde{I}(d) \right)\lra 0.
\end{equation}
We regard $\mc{A}^\bullet$ as a complex with $\mc{A}^0 = V \otimes \mc{O}_{\PP^n}$ and $\mc{A}^{-i} = W_i (-i)$ for $1 \le i \le p$. Note that $\mc{A}^\bullet$ is exact in cohomological degrees $-p+1,\dots,-1$. Since $I$ is $\m$-primary, we have $\widetilde{I}(d) = \mc{O}_{\PP^n}(d)$. The map $d_0$ in (\ref{eq:linearStrand}) determines a linear series  $(V,\mc{O}_{\PP^n}(d))$. We define the \textbf{$i$-th syzygy bundle} by
\[E_i \coloneq \ker(d_i).\]
When $i = 0$, $E_0 = \ker(d_0) = M_V$ is the classical syzygy bundle of the linear series $(V,\mc{O}_{\PP^n}(d))$, see \cite{L}*{Section~1.3}. Throughout the paper, we exploit the connection between the homological properties of $I$ and the sheaf cohomological properties of $E_i$.
\begin{example}\label{ex:Idn}
    One important class of ideals with partially linear virtual resolution arises naturally in the study of binary forms, see \cite{RSWY}. Let $S = \Sym(\Sym^n (\CC^2))$, where $\Sym^n (\CC^2)$ is identified with the space of binary forms of degree $n$, and let $I_{d,n}$ be the ideal generated by $\Sym^{nd}(\CC^2) \subseteq \Sym^d(\Sym^n (\CC^2)) = S_d$. The resolution of $I_{d,n}$ is virtually linear for one step, see \cite{BFL}, but not linearly presented in general. Theorem~\ref{thm:linearBound} and Theorem~\ref{thm:regularityBound} imply that
    \[I_{d,n}^t = \m^{td} \quad \text{for all }t \ge \left\lfloor\frac{3n}{2}\right\rfloor-1 \quad \text{and} \quad \reg(S/I_{d,n}) \le \left\lceil \frac{n+1}{2}\right\rceil(d-1).\]
    In fact, for this specific class, \cite{RSWY} proves the stronger results
    \[I_{d,n}^t = \m^{td} \quad \text{for all }t \ge n \quad \text{and} \quad \reg(S/I_{d,n}) = \left\lceil \frac{n+1}{2}\right\rceil(d-1).\]
\end{example}

\subsection*{Acknowledgment}
I would like to thank Claudiu Raicu for his guidance and support throughout this project. Computations with Macaulay2 \cite{GS} provided many valuable insights. The author acknowledges support from the Arthur J. Schmitt Fellowship and the National Science Foundation Grant DMS-2302341.

\section{Preliminaries}\label{sec:prelim}
\subsection{Notation and Conventions}
Throughout the paper, $S = \CC[x_0,\dots,x_n]$ is the homogeneous coordinate ring of $\PP^n$, $\m$ is the maximal ideal generated by the variables, and $I$ is an $\m$-primary ideal generated in degree $d$. For a coherent sheaf $\mc{F}$ on $\PP^n$, we denote its $i$-th sheaf cohomology by $H^i(\mc{F}) = H^i(\PP^n,\mc{F})$. Given a complex $\mc{C}^\bullet$ of coherent sheaves on $\PP^n$, the $i$-th cohomology sheaf of $\mc{C}^\bullet$ is denoted by $\mc{H}^i(\mc{C}^\bullet)$ and the $i$-th hypercohomology of $\mc{C}^\bullet$ is denoted by $\mathbb{H}^i(\mc{C}^{\bullet})$.

\subsection{Complexes and Hypercohomology} \label{sec:prelim-hyper}
Let $\mc{C}^\bullet \colon \cdots \lra \mc{C}^{-1} \lra \mc{C}^0 \lra 0$ be a complex of coherent sheaves on $\PP^n$. There are two spectral sequences $\prescript{I}{}E,\prescript{II}{}E$ with differentials
\[\prescript{I}{}d_i \colon \prescript{I}{}E_i^{p,q} \lra \prescript{I}{}E_i^{p+i,q-i+1},\quad \prescript{II}{}d_i \colon \prescript{II}{}E_i^{p,q} \lra \prescript{II}{}E_i^{p-i+1,q+i}\]
both converging to the hypercohomology of $\mc{C}^\bullet$, with $E_2$-pages given by
\[\prescript{I}{}E^{p,q}_2 = \mc{H}^p(H^q(\mc{C}^\bullet)), \quad \prescript{II}{}E^{p,q}_2 = H^q(\mc{H}^p(\mc{C}^\bullet)) \quad \implies \mathbb{H}^{p+q}(\mc{C}^\bullet).\]
For details on spectral sequences, we refer the reader to \cite{Weibel}*{Section~5.7}. We define the \textbf{augmented complex} $\widehat{\mc{C}}^{\bullet}$ by setting $\widehat{\mc{C}}^{\bullet}[1]$ to be the mapping cone of the natural map $\mc{C}^\bullet \lra \mc{H}^0(\mc{C}^\bullet)$, where $\mc{H}^0(\mc{C}^\bullet)$ is regarded as a complex concentrated in cohomological degree $0$. Concretely,
\[\widehat{\mc{C}}^{\bullet} \colon \cdots \lra \mc{C}^{-1} \lra \mc{C}^{0} \overset{\theta}{\lra } \mc{H}^0(\mc{C}^{\bullet}) \lra 0.\]
Note that the augmented complex $\widehat{\mc{C}}^\bullet$ is exact in cohomological degrees $0$ and $1$. We denote by $H^0(\theta) \colon H^0(\mc{C}^0) \lra H^0(\mc{H}^0(\mc{C}^\bullet))$ the map on global sections induced by $\theta$.
\begin{lemma}\label{lem:hypercohomology}
    If $H^i\left(\widehat{\mc{C}}^{-i}\right) = 0$ for all $i \ge 1$ and $H^{i+1}\left(\widehat{\mc{C}}^{-i}\right)  = 0$ for all $i \ge 0$, then
    \[\coker(H^0(\theta)) = \mathbb{H}^1\left( \widehat{\mc{C}}^\bullet\right).\]
    In particular, these hypotheses hold when $\mc{C}^{i} = V_i\otimes \mc{O}_{\PP^n}(i + j)$ for some vector spaces $V_i$ and a non-negative integer $j$.
\end{lemma}
\begin{proof}
    Consider the spectral sequence $\prescript{I}{}E$. It follows from the assumption $H^{i+1}\left(\widehat{\mc{C}}^{-i}\right) = 0$ for all $i \ge 0$ that $\prescript{I}{}E_2^{-i,i+1} = 0$ for all $i \ge 0$. In particular, we have $\mathbb{H}^1\left( \widehat{\mc{C}}^\bullet\right) = \prescript{I}{}E_{\infty}^{1,0}$. Since $H^{i}\left(\widehat{\mc{C}}^{-i}\right)  = 0$ for all $i \ge 1$, $\prescript{I}{}d_{i+1} \colon \prescript{I}{}E_{i+1}^{-i,i} \lra \prescript{I}{}E_{i+1}^{1,0}$ is the zero map for all $i \ge 1$. Hence,
    \[\coker(H^0(\theta)) =\prescript{I}{}E_2^{1,0} = \prescript{I}{}E_\infty^{1,0} =  \mathbb{H}^1\left( \widehat{\mc{C}}^\bullet\right).\]
    For the second part, the terms of ${\mc{C}}^\bullet$ are direct sums of line bundles, so their sheaf cohomology can be computed using \cite{Hartshorne}*{Theorem~III.5.1}. In particular, we have
    \begin{align*}
        H^i\left(\widehat{\mc{C}}^{-i}\right) = H^i(V_{-i} \otimes \mc{O}_{\PP^n}(-i+j)) = 0\quad &\text{for all } i \ge 1\\
        H^{i+1}\left(\widehat{\mc{C}}^{-i}\right) = H^{i+1}(V_{-i} \otimes \mc{O}_{\PP^n}(-i+j)) = 0\quad &\text{for all } i \ge 0.
    \end{align*}
    This proves the lemma.
\end{proof}

\subsection{Castelnuovo--Mumford regularity}
Let $M$ be a finitely generated graded $S$-module. The \textbf{Castelnuovo--Mumford regularity}, or simply \textbf{regularity}, of $M$ is defined by 
\[\reg(M) = \max\{j \,\colon \Tor_{i}^S(M,\CC)_{i+j} \ne 0 \,\,\text{ for some }i\}.\]
When $M$ has finite length, there is a simpler description. By \cite{DE-syzygies}*{Corollary~4.4},
\begin{equation}\label{eq:regFL}
    \reg(M) = \max \{j \,\colon M_j \ne 0\}.
\end{equation}
\noindent We now prove a regularity bound for powers of $I$ using the classical syzygy bundle.
\begin{lemma}\label{lem:regofPower}The regularity of the $t$-th power of $I$ satisfies
    \[\reg(I^t) \le \reg(I) + (t-1)d.\]
\end{lemma}
\begin{proof}
    Consider the short exact sequence
    \[0 \lra M_V \lra V \otimes \mc{O}_{\PP^n} \lra \mc{O}_{\PP^n}(d) \lra 0.\]
    It follows from the long exact sequence in sheaf cohomology that
    \[H^1(M_V(j)) = \coker(V\otimes S_j \lra S_{j+d})=(S/I)_{j+d} \quad \text{for all }j \ge -d.\]
    By (\ref{eq:regFL}), we have $\reg(I) = \reg(S/I)+1 = \max \{j \,\colon (S/I)_j \ne 0\} + 1$ which implies that
    \begin{equation}\label{eq:vanishingofMv}
        H^1(M_V(j)) = 0 \quad \text{ for all }j \ge \reg(I)-d.
    \end{equation}
    For every $t$, we have the following exact sequence
    \[\mc{C}^\bullet \colon 0 \lra \bigwedge^t M_V \lra \cdots \lra \Sym^{t-1}V \otimes M_V \lra \Sym^t V \otimes \mc{O}_{\PP^n} \overset{\theta}{\lra} \mc{O}_{\PP^n}(td) \lra 0.\]
    Since $\coker(H^0(\theta(j-td)))=(S/I^t)_j$, to show $\reg(I^t) \le \reg(I) + (t-1)d$, by (\ref{eq:regFL}), it suffices to show that
    \[\coker(H^0(\theta(j-td))) = 0 \quad \text{for all }j \ge \reg(I) + (t-1)d.\]
    To apply Lemma~\ref{lem:hypercohomology}, it remains to show that
    \begin{equation}\label{eq:vanishingofMv1}
        H^i(\wedge^i M_V(j)) = 0 \quad \text{ for all }j \ge \reg(I)-d, i\ge 1, \text{ and}
    \end{equation}
    \begin{equation}\label{eq:vanishingofMv2}
        H^{i+1}(\wedge^i M_V(j)) = 0 \quad \text{ for all }j \ge \reg(I)-d, i\ge 0.
    \end{equation}
    We prove (\ref{eq:vanishingofMv1}) and (\ref{eq:vanishingofMv2}) by induction. The base case of (\ref{eq:vanishingofMv1}) follows from (\ref{eq:vanishingofMv}), and the base case of (\ref{eq:vanishingofMv2}) follows from the fact that $H^1(\mc{O}_{\PP^n}(j)) = 0$ for all $j \ge \reg(I)-d \ge 0$. We now consider the long exact sequence in sheaf cohomology induced by
    \[0 \lra \bigwedge^i M_V \lra \bigwedge^i V \otimes \mc{O}_{\PP^n} \lra \bigwedge^{i-1}M_V (d) \lra 0.\]
    Since $H^i(\mc{O}_{\PP^n}(j)) = 0$ for all $i \ge 1$ and $j \ge 0$, we have 
    \begin{align*}
        &H^i(\wedge^i M_V(j)) = H^{i-1}(\wedge^{i-1} M_V(j+d)) \quad \text{ for all }i \ge 2 \text{ and } j \ge \reg(I)-d \ge 0, \text{ and}\\
        &H^{i+1}(\wedge^i M_V(j)) = H^{i}(\wedge^{i-1} M_V(j+d)) \quad \text{ for all }i \ge 1 \text{ and } j \ge \reg(I)-d \ge 0.
    \end{align*}
    Hence, this proves (\ref{eq:vanishingofMv1}) and (\ref{eq:vanishingofMv2}) by induction. To conclude the proof, we apply Lemma~\ref{lem:hypercohomology} and get $\coker(H^0(\theta(j-td))) = \mathbb{H}^1(\mc{C}^\bullet) = 0$ for all $j \ge \reg(I) + (t-1)d$.
\end{proof}
\subsection{Symmetric and exterior products of complexes} In this section, we recall some useful facts about the symmetric and exterior products of complexes in our specific setting. For a more general discussion of Schur complexes, see \cite{weyman}*{Section~2.4}. Let $\mc{Q}_i^r(-)$ be $\Sym^r(-)$ when $i$ is odd and $\bigwedge^r(-)$ when $i$ is even, and let $\mc{P}_i^r(-)$ be $\Sym^r(-)$ when $i$ is even and $\bigwedge^r(-)$ when $i$ is odd. Let
\[\mc{C}^\bullet \colon 0 \lra \mc{C}^{-m} \lra \cdots \lra \mc{C}^{-1} \lra \mc{C}^0 \lra 0\]
be a complex of vector bundles on $\PP^n$ where
\[\mc{H}^i(\mc{C}^\bullet) = \begin{cases}
    \mc{E} & i = -m,\\
    \mc{F} & i = 0,\\
    0 & \text{otherwise,}
\end{cases}\]
for some vector bundles $\mc{E},\mc{F}$ on $\PP^n$. We first describe the terms of symmetric and exterior products of $\mc{C}^\bullet$. We have
\begin{equation}\label{eq:termsSym}
    \left(\Sym^r(\mc{C}^\bullet)\right)^{i} = \bigoplus_{\substack{a_0,\dots,a_m \ge 0\\ a_0 + \cdots + a_m = r\\ a_1 + 2a_2 + \cdots + ma_m = -i}} \bigotimes_{j = 0}^m \mc{P}^{a_j}_{j} (\mc{C}^{-j}).
\end{equation}
\begin{equation}\label{eq:termsWedge}
    \left(\bigwedge^r(\mc{C}^\bullet)\right)^{i} = \bigoplus_{\substack{a_0,\dots,a_m \ge 0\\ a_0 + \cdots + a_m = r\\ a_1 + 2a_2 + \cdots + ma_m = -i}} \bigotimes_{j = 0}^m \mc{Q}^{a_j}_{j} (\mc{C}^{-j}).
\end{equation}

We now describe their cohomology sheaves. Since $\mc{E}$ and $\mc{F}$ are vector bundles, \cite{Weibel}*{Theorem~3.6.3} gives
\[\mc{H}^i((\mc{C}^\bullet)^{\otimes r}) = \begin{cases}
    \bigoplus_{I \subseteq \{1,\dots,r\}, |I| = i'}\mc{E}^{\otimes i'} \otimes \mc{F}^{\otimes (r-i')}& i = -mi',\\
    0& \text{otherwise}.
\end{cases}\]
Consider the $\mathbb{Z}/2$-grading on $\mc{C}^\bullet$ given by
\[F_0 = \bigoplus_{i\text{ even}} \mc{C}^{-i} \quad \text{and}\quad F_1 = \bigoplus_{i\text{ odd}} \mc{C}^{-i}.\]
The usual $\mathbb{Z}/2$-graded-commutative $\mathfrak{S}_r$-action on the tensor product $(\mc{C}^\bullet)^{\otimes r}$ commutes with the differentials, see \cite{weyman}*{page~74}. Since we work over $\mathbb C$, the invariant
functor \((-)^{\mathfrak S_r}\) is exact. Hence, taking
\(\mathfrak S_r\)-invariants commutes with cohomology. Therefore,
\[
\mathcal H^i\!\left(\operatorname{Sym}^r(\mc{C}^\bullet)\right)
=
\mathcal H^i\!\left(((\mc{C}^\bullet)^{\otimes r})^{\mathfrak S_r}\right)
=
\mathcal H^i\!\left((\mc{C}^\bullet)^{\otimes r}\right)^{\mathfrak S_r}.
\]
It follows that
\begin{equation}\label{eq:cohomologySheavesSym}
    \mc{H}^i\left(\Sym^r(\mc{C}^\bullet)\right) = \begin{cases}
    \mc{P}_m^{i'}(\mc{E})\otimes \Sym^{r-i'}(\mc{F})& i = -mi',\\
    0& \text{otherwise}.
\end{cases}
\end{equation}
Similarly, for exterior products, we have
\begin{equation}\label{eq:cohomologySheavesWedge}
    \mc{H}^i\left(\bigwedge^r(\mc{C}^\bullet)\right) = \begin{cases}
    \mc{Q}_m^{i'}(\mc{E})\otimes \bigwedge^{r-i'}(\mc{F})& i = -mi',\\
    0& \text{otherwise}.
\end{cases}
\end{equation}

\section{A Regularity bound}\label{sec:regularityBound}
Recall the complex $\mc{A}^\bullet$ from (\ref{eq:linearStrand}), and let $\widehat{\mc{A}}^\bullet$ be its augmented complex as defined in Section~\ref{sec:prelim-hyper}. Since $\widehat{\mc{A}}^\bullet \otimes \mc{O}_{\PP^n}(j)$ has cohomology sheaves given by
\[\mc{H}^i(\widehat{\mc{A}}^\bullet \otimes \mc{O}_{\PP^n}(j)) = \begin{cases}
    E_p(j) & i = -p,\\
    0 & \text{otherwise,}
\end{cases}\]
it follows from the spectral sequence $\prescript{II}{}E$ that $H^{p+1}(E_p(j)) = \mathbb{H}^1(\widehat{\mc{A}}^\bullet \otimes \mc{O}_{\PP^n}(j))$ for all $j \ge 0$. By Lemma~\ref{lem:hypercohomology}, for all $j \ge 0$, we have
\begin{equation}\label{eq:sheafCoh=S/I}
    H^{p+1}(E_p(j)) =\mathbb{H}^1(\widehat{\mc{A}}^\bullet \otimes \mc{O}_{\PP^n}(j))= \coker(V \otimes S_j \lra S_{j+d}) = (S/I)_{j+d}.
\end{equation}
In particular, when $p = n$, we have
\[(S/I)_d = H^{n+1}(E_n) = 0.\]
It follows that if the resolution of $I$ is virtually linear for $n$ steps, then $I = \m^d$. In the remainder of the paper, we exclude the case $p = n$ and assume throughout that $1 \le p \le n-1$ and $I \ne \m^d$. Hence, by (\ref{eq:regFL}) and (\ref{eq:sheafCoh=S/I}), we have
\[\reg(S/I) = \max \{j \,\colon (S/I)_j \ne 0\} = \max\{j\,\colon H^{p+1}(E_p(j)) \ne 0\} + d.\]
This shows that upper bounds for $\reg(S/I)$ are governed by the vanishing of $H^{p+1}(E_p(j))$,
\begin{equation}\label{eq:cohomologyCriterion}
    \reg(S/I) \le k \quad \iff \quad H^{p+1}(E_p(j)) = 0 \quad \text{for all }j \ge k-d+1.
\end{equation}
Let $\mc{W}^\bullet_l$ denote the $l$-fold exterior product of $\mc{A}^\bullet$.
\begin{lemma}\label{lem:regularityClaim}
    If $\mathbb{H}^{l-1}(\mc{W}^\bullet_l\otimes\mc{O}_{\PP^n}(j)) = 0$, then there is an injection
    \[H^{(l-1)(p+1)}(\mc{Q}^{l-1}_p(E_p)(j+d)) \subseteq H^{l(p+1)}(\mc{Q}^l_p(E_p)(j)).\]
\end{lemma}
\begin{proof}
     It follows from (\ref{eq:termsWedge}) and (\ref{eq:cohomologySheavesWedge}) that the complex $\mc{W}^\bullet_l$ satisfies
    \begin{align*}
        &\mc{W}^i_l = 0 \quad \text{for all }i< -lp\text{ and }i>0,\\
        &\mc{H}^i(\mc{W}^\bullet_l) = \begin{cases}
            \mc{Q}^l_p(E_p) &i = -lp\\
            \mc{Q}^{l-1}_p(E_p)(d) &i = -(l-1)p\\
            0 &\text{otherwise}.
        \end{cases}
    \end{align*}
    Since $\mc{W}_l^\bullet$ is exact except in two cohomological degrees, the only possible nonzero differentials in the spectral sequence $\prescript{II}{}{}E$ are of the form
    \[\prescript{II}{}{}d_{p+1} \colon \prescript{II}{}E_{p+1}^{-(l-1)p,q} \lra \prescript{II}{}E_{p+1}^{-lp,q+p+1} \quad \text{for some }q.\]
    In particular, the assumption that $\mathbb{H}^{l-1}(\mc{W}^\bullet_l\otimes\mc{O}_{\PP^n}(j)) = 0$ implies that $\prescript{II}{}E_{\infty}^{-(l-1)p,(l-1)(p+1)}=\prescript{II}{}E_{p+2}^{-(l-1)p,(l-1)(p+1)} = 0$. Equivalently, the map
    \[\begin{array}{cccc}
       \prescript{II}{}{}d_{p+1} \colon & \prescript{II}{}E_{p+1}^{-(l-1)p,(l-1)(p+1)} & \lra &  \prescript{II}{}E_{p+1}^{-lp,l(p+1)} \\
        & \|& & \|\\
        &  H^{(l-1)(p+1)}(\mc{Q}^{l-1}_p(E_p)(j+d)) & &H^{l(p+1)}(\mc{Q}^l_p(E_p)(j))
    \end{array}\]
    is injective as desired.
\end{proof}

\begin{lemma}\label{lem:cohomologyLadder}
    For $t = \left\lceil \frac{n+1}{p+1}\right\rceil - 1$, we have an injection
    \[H^{p+1}(E_p(j)) \subseteq H^{t(p+1)}(\mc{Q}_p^t(E_p)(j-(t-1)d)) \quad \text{for all }j.\]
\end{lemma}
\begin{proof}
    It suffices to show that for every $2 \le l < (n+1)/(p+1)$, we have 
    \begin{equation}\label{eq:cohoLadderClaim}
        H^{(l-1)(p+1)}(\mc{Q}_p^{l-1}(E_p)(j+d)) \subseteq H^{l(p+1)}(\mc{Q}_p^l(E_p)(j)) \quad \text{for all }j.
    \end{equation}
    By Lemma~\ref{lem:regularityClaim}, it suffices to show that
    \begin{equation}\label{eq:vanishingHypercohomologyRegularity}
        \mathbb{H}^{l-1}(\mc{W}^\bullet_l\otimes\mc{O}_{\PP^n}(j)) = 0 \quad \text{for all }2 \le l < (n+1)/(p+1).
    \end{equation}
    Indeed, since the terms in $\mc{W}^\bullet_l$ are direct sums of line bundles by (\ref{eq:termsWedge}), it follows from \cite{Hartshorne}*{Theorem~III.5.1} and the spectral sequence $\prescript{I}{}E$ that the only cohomological degrees in which the hypercohomology may be nonzero for $\mc{W}^\bullet_l\otimes\mc{O}_{\PP^n}(j)$ are
    \[-lp,-lp+1,\dots,-1,0,n-lp,n-lp+1,\dots,n.\]
    Since $2 \le l < (n+1)/(p+1)$, we have 
    \[0 < l-1 < n-lp.\]
    This concludes the proof.
\end{proof}
\begin{proof}[Proof of Theorem~\ref{thm:regularityBound}]
    By Lemma~\ref{lem:cohomologyLadder} and (\ref{eq:cohomologyCriterion}), it suffices to show
    \begin{equation}\label{eq:regBoundClaim}
        H^{t(p+1)}(\mc{Q}^t_p(E_p)(j)) = 0 \quad \text{for all }j \ge d - t \quad \text{where }t = \left\lceil \frac{n+1}{p+1}\right\rceil - 1.
    \end{equation}
    Consider $\mc{W}^\bullet_{t+1}$, the $(t+1)$-fold exterior product of $\mc{A}^\bullet$. We claim that
    \begin{equation}\label{eq:regularityClaimVanishing}
        \mathbb{H}^{t}(\mc{W}^\bullet_{t+1}\otimes\mc{O}_{\PP^n}(j-d)) = 0
    \end{equation}
    Assuming (\ref{eq:regularityClaimVanishing}) for the moment, it follows from Lemma~\ref{lem:regularityClaim} and Grothendieck vanishing that
    \[H^{t(p+1)}(\mc{Q}^t_p(E_p)(j)) \subseteq  H^{(t+1)(p+1)}(\mc{Q}^{t+1}_p(E_p)(j-d)) = 0\]
    because $(t+1)(p+1) > n$. We now prove the claim (\ref{eq:regularityClaimVanishing}). By (\ref{eq:termsWedge}), the terms of $\mc{W}^\bullet_{t+1}$ are direct sums of line bundles. In particular, they have vanishing intermediate sheaf cohomology. Consider the spectral sequence $\prescript{I}{}{}E$. It follows that $\mathbb{H}^{t}(\mc{W}^\bullet_{t+1}\otimes\mc{O}_{\PP^n}(j-d))$ is a subquotient of $H^n(\mc{W}^{t-n}_{t+1}\otimes\mc{O}_{\PP^n}(j-d))$. Hence, it suffices to show $H^n(\mc{W}^{t-n}_{t+1}\otimes\mc{O}_{\PP^n}(j-d)) = 0$. Since $\mc{W}^{t-n}_{t+1}$ is a direct sum of $\mc{O}_{\PP^n}(t-n)$ by (\ref{eq:termsWedge}) and $t-n+j-d \ge -n$, we have
    \[H^n(\mc{W}^{t-n}_{t+1}\otimes\mc{O}_{\PP^n}(j-d)) = 0.\]
    This concludes the proof.
\end{proof}
\begin{example}\label{ex:sharp}
    We present examples showing that Theorem~\ref{thm:regularityBound} is sharp in several cases.
    \begin{enumerate}
        \item $p = 0$. \quad Take $I = (x_0^d,\cdots,x_n^d)$.
        \item $p = 1$. \quad Take $I = I_{d,n}$ from Example~\ref{ex:Idn}.
        \item $ (n-1)/2 \le p \le n-1$. \quad Let $C \subseteq \PP(H^0(C,L))$ be a smooth projective curve of genus $g = n - p$ embedded via the complete linear series of a line bundle $L$ such that
        \[\deg(L) = 2g + 2(d-2) + 1 + p.\]
        By \cite{ENP}*{Theorem~1.2}, the $(d-2)$-nd secant variety $\Sigma_{d-2}(C)$ of $C$ is arithmetically Cohen-Macaulay. Take $I$ to be an Artinian reduction of the ideal $I(\Sigma_{d-2}(C))$. $I$ is generated in degree $d$, and by \cite{ENP}*{Theorem~1.2}, the resolution of $I$ is linear for $p$ steps with $\reg(S/I) = 2d-2$.
        \item $p = n$. \quad Take $I = \m^d$.
    \end{enumerate}
\end{example}

\section{Powers of an ideal}
Let $\mc{S}^\bullet_t$ be the $t$-fold symmetric product of $\mc{A}^\bullet$ from (\ref{eq:linearStrand}), and $\widehat{\mc{S}}^\bullet_t$ its augmented complex.
\begin{lemma}\label{lem:H1Sk}
    If $t \ge (n/p) - 1$, then $\mathbb{H}^1(\widehat{\mc{S}}^\bullet_{t+1}\otimes \mc{O}_{\PP^n}(j)) = \mathbb{H}^1(\widehat{\mc{S}}^\bullet_{t}\otimes \mc{O}_{\PP^n}(j+d))$ for all $j$.
\end{lemma}
\begin{proof}
Let $\mc{B}^\bullet$ be the following complex
\[0 \lra W_{p}(-p) \lra \cdots \lra W_1(-1) \lra M_V \lra 0.\] 
Note that the only nonzero cohomology sheaves for $\mc{A}^\bullet$ and $\mc{B}^\bullet$ are 
\[\mc{H}^{-p}(\mc{A}^\bullet) = \mc{H}^{-p}(\mc{B}^\bullet) = E_{p} \quad\text{and}\quad  \mc{H}^0(\mc{A}^\bullet) = \mc{O}_{\PP^n}(d).\]
Consider the short exact sequence
\begin{equation}\label{eq:BAses}
    0 \lra \mc{B}^\bullet \lra \mc{A}^\bullet \lra \mc{O}_{\PP^n}(d) \lra 0.
\end{equation}
For every $r \ge 1$, we can take the $r$-fold symmetric product of the inclusion $\mc{B}^\bullet \lra \mc{A}^\bullet$ and get the following short exact sequence
\begin{equation}\label{eq:powersLemmaSym}
    0 \lra \Sym^{r}(\mc{B}^\bullet) \lra \mc{S}^\bullet_{r} \lra \mc{S}^\bullet_{r-1} \otimes \mc{O}_{\PP^n}(d) \lra 0.
\end{equation}
It follows from (\ref{eq:cohomologySheavesSym}) that $\Sym^{r}(\mc{B}^\bullet)$ has cohomology
\[\mc{H}^i(\Sym^{r}(\mc{B}^\bullet)) = \begin{cases}
    \mc{P}_{p}^{r}(E_{p}) & i = -rp,\\
    0 &\text{otherwise},
\end{cases}\]
In particular, $\mc{H}^0(\Sym^{r}(\mc{B}^\bullet)) = \mc{H}^1(\Sym^{r}(\mc{B}^\bullet)) = 0$. It follows that (\ref{eq:powersLemmaSym}) extends to a short exact sequence of augmented complexes
\[0 \lra \Sym^{r}(\mc{B}^\bullet) \lra \widehat{\mc{S}}^\bullet_{r} \lra \widehat{\mc{S}}^\bullet_{r-1} \otimes \mc{O}_{\PP^n}(d) \lra 0.\]
The corresponding long exact sequence in hypercohomology gives $\mathbb{H}^1(\widehat{\mc{S}}^\bullet_{t+1}\otimes \mc{O}_{\PP^n}(j)) = \mathbb{H}^1(\widehat{\mc{S}}^\bullet_{t}\otimes \mc{O}_{\PP^n}(j+d))$ provided that
\begin{equation}\label{eq:vanishingCohomologyClaim}
    \mathbb{H}^1(\Sym^{t+1}(\mc{B}^\bullet)\otimes \mc{O}_{\PP^n}(j)) = \mathbb{H}^2(\Sym^{t+1}(\mc{B}^\bullet)\otimes \mc{O}_{\PP^n}(j)) = 0.
\end{equation}
We now prove (\ref{eq:vanishingCohomologyClaim}). 
Since $\Sym^{t+1}(\mc{B}^\bullet)$ is exact except in one cohomological degree, it follows from the spectral sequence $\prescript{II}{}{}E$ that
\begin{align*}
    \mathbb{H}^1(\Sym^{t+1}(\mc{B}^\bullet)\otimes \mc{O}_{\PP^n}(j)) &= H^{(t+1)p+1}(\mc{P}_{p}^{t+1}(E_{p}) (j)),\\
    \mathbb{H}^2(\Sym^{t+1}(\mc{B}^\bullet)\otimes \mc{O}_{\PP^n}(j)) &= H^{(t+1)p+2}(\mc{P}_{p}^{t+1}(E_{p}) (j)).
\end{align*}
By Grothendieck vanishing, since $(t+1)p \ge n$, we have
\[H^{(t+1)p+1}(\mc{P}_{p}^{t+1}(E_{p}) (j)) = H^{(t+1)p+2}(\mc{P}_{p}^{t+1}(E_{p}) (j)) = 0.\]
This proves (\ref{eq:vanishingCohomologyClaim}) and concludes the proof.
\end{proof}
\begin{proof}[Proof of Theorem~\ref{thm:linearBound}]
    By (\ref{eq:termsSym}), the terms of $\mc{S}_r^\bullet$ are given by
    \[\mc{S}_r^i = V_i \otimes \mc{O}_{\PP^n}(i)\quad \text{where } V_i = \bigoplus_{\substack{a_0,\dots,a_p \ge 0\\ a_0 + \cdots + a_p = r\\ a_1 + 2a_2 + \cdots + pa_p = -i}} \bigotimes_{j = 0}^p \mc{P}_j^{a_j}(\mc{A}^{-j}).\]
    It follows from Lemma~\ref{lem:hypercohomology} that
    \[(S/I^r)_{rd} = \coker(\Sym^r V \lra S_{rd}) = \mathbb{H}^1(\widehat{\mc{S}}^\bullet_{r}).\]
Since both $I^r$ and $\m^{rd}$ are generated in degree $rd$, it suffices to show $(I^r)_{rd} = (\m^{rd})_{rd} = S_{rd}$ which is equivalent to $(S/I^r)_{rd} = 0$. Hence, it suffices to show that for $t \ge  \left\lfloor\frac{n}{p+1}\right\rfloor + \left\lfloor\frac{n-1}{p}\right\rfloor $, we have $\mathbb{H}^1(\widehat{\mc{S}}^\bullet_{t})  = 0$. Let $k = \left \lceil \frac{n-p}{p}\right \rceil = \left\lfloor\frac{n-1}{p}\right\rfloor$, by Lemma~\ref{lem:H1Sk}, we have
    \begin{equation}\label{eq:hypercohomologyEquality}
        \mathbb{H}^1(\widehat{\mc{S}}^\bullet_{t})  = \mathbb{H}^1(\widehat{\mc{S}}^\bullet_{t-1}(d)) = \cdots = \mathbb{H}^1\left(\widehat{\mc{S}}^\bullet_{k}\left((t-k)d\right)\right). 
    \end{equation}
    It follows from Lemma~\ref{lem:hypercohomology} that
    \begin{equation}\label{eq:HypercohomologyToS/Ik}
        \mathbb{H}^1\left(\widehat{\mc{S}}^\bullet_{k}\left((t-k) d\right)\right) =\coker(\Sym^k V \otimes S_{(t-k)d} \lra S_{td}) =  (S/I^k)_{td}.
    \end{equation}
    By Theorem~\ref{thm:regularityBound} and Lemma~\ref{lem:regofPower}, we have
    \[\reg(I^k) \le \reg(I) + (k-1)d \le \left \lceil\frac{n+1}{p+1}\right\rceil (d - 1) + 1 + (k-1)d\le \left(\left\lfloor\frac{n}{p+1}\right\rfloor + k \right) d \le td.\]
    By (\ref{eq:regFL}), (\ref{eq:hypercohomologyEquality}), and (\ref{eq:HypercohomologyToS/Ik}), we have
    \[\mathbb{H}^1(\widehat{\mc{S}}^\bullet_{t}) = \mathbb{H}^1\left(\widehat{\mc{S}}^\bullet_{k}\left((t-k) d\right)\right) = (S/I^k)_{td} = 0.\]
    This proves the theorem.
\end{proof}

%\section{Minimal generators of $I$}

\end{document}